\documentclass{article}
\usepackage[utf8]{inputenc}
\usepackage[T1]{fontenc} 
\usepackage{csquotes} 
\usepackage[doi=false, isbn=false, eprint=false, giveninits=true
, style=authoryear, maxcitenames=2, citetracker=true]{biblatex}
\AtEveryCitekey{\ifciteseen{}{\defcounter{maxnames}{99}}}
\DeclareNameAlias{sortname}{last-first}

\usepackage{amsmath}
\usepackage{amssymb} 
\usepackage{mathrsfs} 
\usepackage{bbm} 
\usepackage{amsthm} 
\usepackage{mathtools}
\usepackage{xcolor}
\usepackage{blindtext}
\usepackage[shortlabels]{enumitem}
\usepackage{lipsum}
\usepackage{ifthen}
\usepackage{graphicx}
\usepackage{tikz-cd}
\usepackage{hyperref}

\hypersetup{
	colorlinks,
	linkcolor={red!60!black},
	citecolor={green!50!black},
	urlcolor={blue!60!black}
}

\DeclareFieldFormat{postnote}{#1} 

\DeclareDelimFormat{finalnamedelim}{\addspace\&\space}
\DeclareDelimFormat[textcite]{finalnamedelim}{%
	\ifnumgreater{\value{liststop}}{2}{\finalandcomma}{}%
	\addspace\bibstring{and}\space}

\renewbibmacro*{cite}{%
	\printtext[bibhyperref]{%
		\iffieldundef{shorthand}
		{\ifthenelse{\ifnameundef{labelname}\OR\iffieldundef{labelyear}}
			{\usebibmacro{cite:label}%
				\setunit{\printdelim{nonameyeardelim}}}
			{\printnames{labelname}%
				\setunit{\printdelim{nameyeardelim}}}%
			\usebibmacro{cite:labeldate+extradate}}
		{\usebibmacro{cite:shorthand}}}}

\renewbibmacro*{citeyear}{%
	\printtext[bibhyperref]{%
		\iffieldundef{shorthand}
		{\iffieldundef{labelyear}
			{\usebibmacro{cite:label}}
			{\usebibmacro{cite:labeldate+extradate}}}
		{\usebibmacro{cite:shorthand}}}}

\renewbibmacro*{textcite}{%
	\ifnameundef{labelname}
	{\iffieldundef{shorthand}
		{\printtext[bibhyperref]{%
				\usebibmacro{cite:label}}%
			\setunit{%
				\global\booltrue{cbx:parens}%
				\printdelim{nonameyeardelim}\bibopenparen}%
			\ifnumequal{\value{citecount}}{1}
			{\usebibmacro{prenote}}
			{}%
			\printtext[bibhyperref]{\usebibmacro{cite:labeldate+extradate}}}
		{\printtext[bibhyperref]{\usebibmacro{cite:shorthand}}}}
	{\printtext[bibhyperref]{\printnames{labelname}}%
		\setunit{%
			\global\booltrue{cbx:parens}%
			\printdelim{nameyeardelim}\bibopenparen}%
		\ifnumequal{\value{citecount}}{1}
		{\usebibmacro{prenote}}
		{}%
		\usebibmacro{citeyear}}}

\renewbibmacro*{cite:shorthand}{%
	\printfield{shorthand}}

\renewbibmacro*{cite:label}{%
	\iffieldundef{label}
	{\printfield[citetitle]{labeltitle}}
	{\printfield{label}}}

\renewbibmacro*{cite:labeldate+extradate}{%
	\printlabeldateextra}

\makeatletter
\g@addto@macro\bfseries{\boldmath}
\makeatother

\newtheorem{prop}{Proposition}[section]
\newtheorem{thm}[prop]{Theorem}
\newtheorem{thmx}{Theorem}

\newtheorem{lm}[prop]{Lemma}
\newtheorem{cor}[prop]{Corollary}
\theoremstyle{definition}
\newtheorem{defi}[prop]{Definition}

\newtheorem*{rk}{Remark}

\newcommand{\mc}{\mathcal}
\newcommand{\mb}{\mathbb}
\newcommand{\mf}{\mathfrak}
\newcommand{\mr}{\mathrm}
\newcommand{\mscr}{\mathscr}
\newcommand{\sbullet}{{\scalebox{0.6}{$\bullet$}}}
\newcommand{\eps}{\varepsilon}
\newcommand{\om}{\omega}
\renewcommand{\phi}{\varphi}

\newcommand{\placeholder}{\makebox[1.5ex]{\raisebox{0pt}[1.5ex]{$\cdot$}}}

\newcommand{\dz}{\partial}
\newcommand{\dbz}{\conj \partial}
\newcommand{\dc}{\mr d^{\mr c}}
\newcommand{\ddc}{\dd \dc}

\newcommand{\dd}{\mathrm{d}}

\newcommand{\conj}[1]{\overline{#1}}
\newcommand{\metr}[1]{\overline{#1}}

\newcommand{\Id}{\mathrm{Id}}

\newcommand{\CC}{\mb C}
\newcommand{\PP}{\mb P}
\newcommand{\sbsq}{\subseteq}

\newcommand{\closure}[2][]{\overline{#2}^{#1}}

\newcommand{\quotient}[2]{{\raisebox{.2em}{$#1$}\left/\raisebox{-.2em}{$#2$}\right.}}

\DeclareMathOperator{\Sym}{Sym}

\newcommand{\GL}{\mathrm{GL}}
\newcommand{\norm}[1]{\left\Vert#1\right\Vert}
\newcommand{\va}[1]{\left \vert #1 \right \vert}

\DeclareMathOperator*{\colim}{colim}

\newcommand{\into}{\hookrightarrow}
\newcommand{\onto}{\twoheadrightarrow}
\newcommand{\isom}{\xrightarrow{
		{\raisebox{-0.65ex}{$\sim$}}}}

\DeclareMathOperator{\Spec}{Spec}
\DeclareMathOperator{\Proj}{Proj}

\let\div\relax
\DeclareMathOperator{\div}{div}

\DeclareMathOperator{\ord}{ord}
\newcommand{\CH}{\mathrm{CH}}

\newcommand{\sHom}{\mc{H}\kern -0.75pt om}
\DeclareMathOperator{\Bl}{Bl}
\newcommand{\OO}{\mc O}

\newcommand{\divnu}[1]{\div_{\mathrm{v}}(#1)}
\newcommand{\XX}{\mf X}
\newcommand{\YY}{\mf Y}
\newcommand{\ZZZ}{\mf Z}
\newcommand{\cform}[1]{A_{\mathrm{c}}^{\ifthenelse{\equal{#1}{}}{}{#1,#1}}}
\newcommand{\mform}[1]{\t{A}^{\ifthenelse{\equal{#1}{}}{}{#1,#1}}}
\newcommand{\ccurr}[1]{D_{\mathrm{c}}^{\ifthenelse{\equal{#1}{}}{}{#1,#1}}}
\newcommand{\mcurr}[1]{\t{D}^{\ifthenelse{\equal{#1}{}}{}{#1,#1}}}
\renewcommand{\t}[1]{\widetilde{#1}}
\newcommand{\ddiv}[3][]{%
	\ifthenelse{ \equal{#1}{}}%
	{(\div #2)_{#3}}%
	{(\div_{\mathrm v} #2)_{#3}}%
}
\newcommand{\eval}[1]{E_{#1}}
\newcommand{\proddot}{\cdot}

\makeatletter
\DeclareDocumentCommand{\nameditem}{ o m }{%
	\IfNoValueTF{#1}{\item}{\item[#1]} (#2).\hskip7\p@ plus\p@ minus\p@%
}
\makeatother

\makeatletter 
\newcommand\mynobreakpar{\par\nobreak\@afterheading} 
\makeatother

\title{Non-archimedean Green currents for the zero-locus of a regular section}
\date{\today}
\author{Léo Dubocs}

\newcommand{\address}{{
		\bigskip
		\footnotesize
		Léo Dubocs, \textsc{Universit\'e Paris Cit\'e and Sorbonne Universit\'e, CNRS, IMJ-PRG, F-75013 Paris, France}\par\nopagebreak
		\textit{E-mail address}: \texttt{leo.dubocs@imj-prg.fr}
		}}

\begin{document}
	\maketitle
	\begin{abstract}
		We extend on the work of Bloch, Gillet, Soulé on non-archimedean Arakelov geometry, by proving in this context explicit formulas for Green currents, which are analogs of known formulas in complex geometry. More specifically, we prove an analog of the Poincaré-Lelong formula, as well as an equivalent of a formula by Bost, Gillet, Soulé, which expresses a Green current for the zero-locus of a regular section of a vector bundle. 
		As corollaries, we also obtain non-archimedean counterparts of Levine formula and Martinelli formula.
	\end{abstract}
	\section*{Introduction}
	
	In this article, we work within the framework of non-archimedean Arakelov geometry, as developed by \textcite{bloch_gillet_soule--NonArchimedeanArakelovTheory1995}. The purpose of this paper is to study some analogs of formulas for Green currents from complex geometry. In order to define satisfactory counterparts of the objects under consideration, we also introduce, along the way, new constructions of currents within this non-archimedean framework.
	
	\medskip
	
	Green currents have become important objects in Arakelov geometry since their use by \textcite{gillet_soule--ArithmeticIntersectionTheory1990} to construct arithmetic intersection rings, generalizing Arakelov's work on arithmetic surfaces \parencite{arakelov--IntersectionTheoryDivisorsArithmetic1976}. A Green current for a cycle~$Z$ on a complex variety $X$ is a current~$g$ that satisfies a differential equation of the form $\ddc g + \delta_Z = [\om_Z]$, where~$\delta_Z$ is the current of integration over the cycle, and~$\om_Z$ is a smooth differential form representing the cohomology class of $Z$.
	
	Their existence is guaranteed by means of Hodge theory, and despite their significance, few explicit examples are known. 
	The simplest instance of a Green current is given by the Poincaré-Lelong formula in the case of divisors:
	\begin{prop}[Poincaré-Lelong]
		Let~$X$ be a complex analytic variety, $(L, \norm{\placeholder})$ a metrized line bundle on~$X$ and~$s$ a meromorphic regular section of~$L$ on~$X$. Then 
		\[\ddc [-\log \norm{s}^2] + \delta_{\div s} = [c_1(\metr{L})].\]
	\end{prop}

	Another explicit and important example in higher codimension is given by a result from \textcite{bost_gillet_soule--HeightsProjectiveVarietiesPositive1994}, as follows.
	
	Let~$X$ be a complex manifold, $(F, \norm{\placeholder})$ a metrized vector bundle of rank~$r$ and~$\sigma$ a global section of~$F$ which is regular, that is, $\sigma$ is locally given by a regular sequence of functions.
	Denote by~$\pi \colon \t X \to X$ the blow-up of~$X$ along the closed subset~$Z = \{\sigma = 0\}$, then there is a canonical closed immersion~$f \colon \t X \into \PP(F)$, where $\PP(F) = \Proj(\Sym F^\vee)$ is the projective bundle of lines in $F$. Finally, denote by~$Q_F$ the canonical quotient bundle on~$\PP(F)$, and using the results of \textcite[proposition 4.2]{bott_chern--HermitianVectorBundlesEquidistribution1965}, choose~$\eta$ a smooth $(r-1, r-1)$ differential form on~$\PP(F)$ satisfying 
	\[c_r(p^*\metr{F}) - c_1(\metr{\OO_F(-1)}) \wedge c_{r-1} (\metr{Q_F}) = \ddc \eta.\]
	\begin{thm}[\protect{\cite[§1.2.3]{bost_gillet_soule--HeightsProjectiveVarietiesPositive1994}}]\label{thm:bost-g-s-complex}
		In the above situation, let 
		\[\Lambda = \pi_* \left[ -\pi^* \log \norm{\sigma}^2 \wedge f^* c_{r-1}(\metr{Q_F}) + f^* \eta \right].\]
		Then 
		\[\ddc \Lambda + \delta_Z = [c_r(\metr{F})],\]
		that is, $\Lambda$ is a Green current for~$Z$.
	\end{thm}
	
	Note that if $F$ is actually a line bundle, then we recover the Poincaré-Lelong formula.

	\medskip	
	
	One feature of Arakelov geometry is the striking asymmetry between finite places, where the considerations are algebraic, and infinite places, where the objects involved are the aforementioned Green currents, which are intrinsically objects of analytic nature. In view of reducing this asymmetry, \textcite{bloch_gillet_soule--NonArchimedeanArakelovTheory1995} proposed a non-archimedean version of Arakelov theory. 
	After having fixed an excellent discrete valuation ring $R$ with fraction field $K$, and a smooth projective variety $X$ over $K$, they use bivariant intersection theory on all (regular proper flat) $R$-models of $X$ to define notions of differential forms, currents, metrized line bundles and Green currents on $X$.

	They give an explicit example of Green current by proving a Poincaré-Lelong formula in the case~$L = \OO_X$: in this formula, the analog of~$-\log \va{f}^2$ is a current, denoted by~$\divnu{f}$, which is given on a model $\XX$ of~$X$ by the vertical part of the divisor of~$f$ on~$\XX$. 
	
	\medskip 
	
	In this paper, we extend on this work by proving the general form of the Poincaré-Lelong formula, as well as an analog in this setup of Bost-Gillet-Soulé formula. 
	Our main result is the following non-archimedean version of theorem~\ref{thm:bost-g-s-complex}:
	\begin{thmx}[theorem \ref{thm:bost-g-s} below]\label{thm:bost-g-d-letter}
		Let $X$ be a proper smooth variety over $K$, $\metr{F}$ a metrized vector bundle of rank $r$ and $\sigma$ a regular global section of $F$. 
		Denote by $\pi \colon \t X \to X$ the blow-up of $X$ along the closed subset $Z = \{\sigma = 0\}$, and by~$f$ the canonical closed immersion $f \colon \t X \into \PP(F)$. 
		Finally let $\metr{Q_F}$ be the canonical quotient bundle on $\PP(F)$ with quotient metric, $E$ the exceptional divisor of $\t X$, $s$ the canonical section of $\OO_{\t X}(E)$, and 
		\[\Lambda = \pi_* \left( \divnu{s} \wedge f^* c_{r-1}(\metr{Q_F}) \right).\]
		
		Then 
		\[\ddc \Lambda + \delta_Z = [c_r(\metr{F})],\]
		that is, $\Lambda$ is a Green current for $Z$.
	\end{thmx}
	
	Let us draw attention to two differences with the original statement. First, the non-archimedean formula doesn't have the extra term $f^* \eta$ coming from Bott-Chern theory. This is because non-archimedean Chern forms, as constructed by \textcite{bloch_gillet_soule--NonArchimedeanArakelovTheory1995}, are multiplicative on exact sequences, contrary to the Chern forms appearing in complex geometry. The authors interpret this fact by a ``lack of exactness at infinity''.
	
	The second difference between the complex statement and the non-archimedean one is the appearance of the term $\divnu{s}$ in the latter, in place of $-\pi^*\log \norm{\sigma}^2$. 
	In fact, in the setup of \textcite{bloch_gillet_soule--NonArchimedeanArakelovTheory1995}, we have no notion of differential form on a proper open subset at our disposal, so given a section $\sigma$ of a vector bundle~$F$, we cannot define a differential form that would be the analog of~$-\log\norm{\sigma}^2$. 
	Indeed the right analog is a current, which we describe in section~\ref{sec:-log_sigma}, and which will be denoted by $\divnu{\sigma}$. 
	However we cannot state theorem~\ref{thm:bost-g-s-complex} using $\divnu{\sigma}$, as we cannot multiply two currents. 
	
	To solve this, we note that in the complex setup, we have $\pi^* \sigma = s$, so that 
	\begin{equation}\label{eq:norm_s_sigma}
		-\log \norm{s}^2 = -\pi^* \log \norm{\sigma}^2,
	\end{equation}
	where $s$ is the canonical section of the line bundle $\OO_{\t X}(E)$. This allows us to state theorem~\ref{thm:bost-g-d-letter} using this section $s$.
	
	\medskip
	
	In complex geometry, a special case of theorem~\ref{thm:bost-g-s-complex} had previously been shown by \textcite{griffiths_king--NevanlinnaTheoryHolomorphicMappings1973} when the vector bundle and the metric split as $F = L \oplus \dots \oplus L$, with~$L$ a metrized line bundle:
	\begin{thm}[\protect{\cite[proposition~1.15] {griffiths_king--NevanlinnaTheoryHolomorphicMappings1973}}]
		\label{thm:levine_gh_complex}
		Let~$X$ be a complex analytic variety, $(L,\norm{\placeholder})$ a metrized line bundle on~$X$, $r\geq 2$ an integer and~$s_0, \dots, s_{r-1}$ global sections of~$L$. 
		Assume that the closed set~$Z$ defined by $Z = \{s_0 = \dots = s_{r-1} = 0\}$ has pure codimension~$r$, and let $\om = c_1(\metr{L})$, $\om_0 = \om + \ddc \log \bigl(\sum_k \norm{s_k}^2\bigr)$ and 
		\[\Lambda = -\log \Bigl(\sum_k \norm{s_k}^2\Bigr) \wedge \sum_{k=0}^{r-1} \om^{r-1-k} \wedge \om_0^k.\]
		
		Then 
		\begin{equation*}
			\ddc [\Lambda] + \delta_Z = [\om^r],
		\end{equation*}
		that is, $[\Lambda]$ is a Green current for~$Z$.
	\end{thm}
	
	This theorem is itself a generalization of previous results. In the case where $X = \PP^n$, $L = \OO(1)$ and $s_k = X_k$ the result was previously known under the name of Levine formula (\cite{levine--TheoremHolomorphicMappingsComplex1960}, see also \cite[proposition~5.1] {gillet_soule--CharacteristicClassesAlgebraicVector1990}). 
	This name is nowadays also used to refer to theorem~\ref{thm:levine_gh_complex}.
	
	In the case where $L = \OO_X$, the definition of $\Lambda$ simplifies to 
	\[\Lambda = - \log \norm{f}^2 (\ddc \log \norm{f}^2)^{p-1}\] and the resulting theorem was also previously known under the name of Martinelli formula (\cite{martinelli--AlcuniTeoremiIntegraliFunzioni1938}, see also \cite[proposition~1.10] {griffiths_king--NevanlinnaTheoryHolomorphicMappings1973}).
	
	We prove in theorem~\ref{thm:levine_gh} a non-archimedean analog of theorem~\ref{thm:levine_gh_complex}, as well as of the two aforementioned corollaries. Since the differential form $\om_0$ in the complex theorem is defined only on $U = X \setminus Z$, we must once again slightly rephrase the statement to instead involve currents and the blow-up $\t X$, just as we did with Bost-Gillet-Soulé formula.
	
	\medskip 
	
	This paper is organized as follows.	
	In the first part, we recall the objects and definitions involved in the work of \textcite{bloch_gillet_soule--NonArchimedeanArakelovTheory1995}, and we prove several compatibility lemmas between these objects, including a projection formula in proposition~\ref{prop:projection_formula}. In the second part we discuss metrics on vector bundles and Chern forms, and we generalize in proposition~\ref{prop:poincare_lelong} the Poincaré-Lelong equation from \textcite{bloch_gillet_soule--NonArchimedeanArakelovTheory1995}.
	
	In a third section, we present two new constructions of currents. The first one explains how to construct the infimum of two currents in $\mcurr{0}(X)$. This will later be used to define the current $\inf (\divnu{s_k})$, which we see as the non-archimedean equivalent of $-\log \sum \norm{s_k}^2$. The second one is a gluing construction, allowing us to recover a full current from local data.
	
	In section 4, we state and prove theorem~\ref{thm:bost-g-d-letter}.	
	Then in section 5, we spend some time defining a current $\divnu{\sigma}$, thought of as the analog of~$[-\log \norm{\sigma}^2]$, using the constructions of section \ref{sec:construct}. We show in proposition~\ref{prop:divnusigma_bien_def} that 
	\[\pi_* \divnu{s} = \divnu{\sigma},\]
	justifying our definition in view of~(\ref{eq:norm_s_sigma}).
	
	Finally, in the sixth section we focus on the special case $F = L \oplus \dots \oplus L$ to state and prove a non-archimedean counterpart of theorem~\ref{thm:levine_gh_complex}. This result then implies analogs for both the classical Levine formula and Martinelli formula.
	
	\paragraph{Acknowledgments.} I would like to warmly thank my advisor Antoine Chambert-Loir for his very valuable advice and his careful proofreading of this paper.

	\section{Non-archimedean Arakelov theory \textit{à la} Bloch-Gillet-Soulé}
	
	In this section we mainly recall the definitions and objects introduced in \textcite{bloch_gillet_soule--NonArchimedeanArakelovTheory1995}. 
	We fix an excellent discrete valuation ring~$R$ over which to work, with fraction field~$K$ and residue field~$\kappa$.

	\subsection{Models}\label{sec:models}
	Given~$X$ a proper smooth variety over~$K$, a \emph{model} $\XX$ of~$X$ is a regular proper flat scheme over~$R$ together with an isomorphism $\XX_K \simeq X$ between the generic fiber of $\XX$ and $X$. A \emph{map of models} between $\XX'$ and~$\XX$ is a map of $R$-schemes $\pi \colon \XX' \to \XX$ which preserves the generic fiber.
	Given a model $\XX$ of~$X$, we write~$\XX_0$ for its special fiber, and if $\pi \colon \XX' \to \XX$, we will write $\pi_0$ for the induced map between special fibers.
	
	We shall make the following hypothesis : every proper smooth variety $X$ over~$K$ admits a model, and given two models $\XX$ and $\XX'$ of $X$, there exists a third model $\XX''$ and maps of models $\XX'' \to \XX$ and $\XX'' \to \XX'$. This is slightly weaker than hypotheses (M\textsubscript{1}) and (M\textsubscript{2}) made in \textcite{bloch_gillet_soule--NonArchimedeanArakelovTheory1995}, and it is fulfilled assuming resolution of singularities for excellent schemes over $R$. In particular, this hypothesis is satisfied when $X$ is a curve, or in equal characteristic zero.
	
	Note that the hypothesis implies that the set of models above a given model~$\XX$ is cofinal.

	\subsection{Bivariant classes}\label{sec:chow}
	\textcite{bloch_gillet_soule--NonArchimedeanArakelovTheory1995}  use bivariant intersection theory on all models of $X$, as well as on the special fibers of these models, to construct non-archimedean analogs of differential forms and currents on $X$. 
	
	Note that even though these models are not schemes over a field, the constructions and results of intersection theory apply to them \parencite[§20.1]{fulton--IntersectionTheory1998}. In this paper, we will however use an absolute notion of dimension, in accordance with \textcite{bloch_gillet_soule--NonArchimedeanArakelovTheory1995}. We denote by $Z_k(\placeholder)$ and $\CH_k(\placeholder)$ the group of~$k$-cycles and the corresponding Chow group, respectively.
	
	\medskip 
	
	Let $Y$ and $Z$ be either varieties over a field, or finite-type separated {$R$-schemes}, and let $f \colon Y \to Z $ be a map between them. Recall that a \emph{bivariant class} in~$\CH^l(f)$ is the data, for all $g \colon Z' \to Z$, of maps $\CH_k(Z') \to \CH_{k-l}(Y \times_Z Z')$, satisfying certain axioms, for the details of which see \textcite[chapter 17] {fulton--IntersectionTheory1998}.
	We will refer to these axioms using the numbering system employed in this chapter, of the form (C\textsubscript{1}) or (A\textsubscript{23}).
	
	We write $\CH^\sbullet(Y)$ instead of~$\CH^\sbullet(\Id_Y \colon Y \to Y)$, it has a ring structure given by composition. 
	We also denote by~$\eval{Y}$ the evaluation map $\CH^\sbullet(Y) \to \CH_\sbullet(Y)$ given by $c \mapsto c(\Id_Y, [Y])$. 
	
	If~$\XX$ is a model of a smooth proper variety~$X$ over~$K$, then the map~$\eval{\XX}$ is an isomorphism: we say that there is Poincaré duality on~$\XX$. The regularity of $\XX$ also implies that the bivariant Chow group $\CH^\sbullet(\XX)$ is commutative (for both claims, see \textcite[§1.2]{bloch_gillet_soule--NonArchimedeanArakelovTheory1995} and \textcite[§3]{kleiman_thorup--IntersectionTheoryEnumerativeGeometry1987}).
	However neither of these two facts is true in general for the special fiber~$\XX_0$, as it may fail to be regular.
	
	\subsection{Maps between Chow groups}
	Given $f \colon Y \to Z$ as above, we can construct several operations on Chow rings:
	\begin{itemize}
		\item Without additional hypotheses on $f$, a pullback map $f^* \colon \CH^\sbullet(Z) \to \CH^\sbullet(Y)$ is defined \parencite[§17.2]{fulton--IntersectionTheory1998}.
		\item If~$f$ is a proper morphism, we can define the usual pushforward of cycles $f_* \colon \CH_\sbullet(Y) \to \CH \sbullet(Z)$.
		\item If~$f$ is a lci morphism, then a map $f^! \colon \CH_\sbullet(Z) \to \CH _\sbullet (Y)$ is defined \parencite[§6.6] {fulton--IntersectionTheory1998}. Note that this map defines a bivariant class, which we denote by $[f] \in \CH^\sbullet(f)$. It satisfies $[f]([Z]) = [Y]$ (see \cite[§17.4]{fulton--IntersectionTheory1998}).
		\item Under the two previous hypotheses, we can define $f_! \colon \CH^\sbullet(Y) \to \CH^\sbullet(Z)$ by $(f_!c) \colon \alpha \mapsto  f_*(c \proddot [f] (\alpha))$.
	\end{itemize}
	
	If $X' \to X$ is a morphism of proper smooth $K$-varieties, $\XX$ a model of $X$ and $g \colon \XX' \to \XX$ a model of $X'$ over $\XX$, then $g$ and $g_0 \colon \XX'_0 \to \XX_0$ satisfies the two above conditions \parencite[§1.6] {bloch_gillet_soule--NonArchimedeanArakelovTheory1995}.
	Therefore the four above constructions are well-defined for both $g$ and $g_0$. 
	
	\subsection{Bloch-Gillet-Soulé theory}
		Let~$X$ be a proper smooth variety over~$K$ of dimension~$n$.
		 \textcite{bloch_gillet_soule--NonArchimedeanArakelovTheory1995} propose the following definitions: 
	\begin{description} 
		\item \textit{The closed forms of type $(k,k)$ on $X$}: \[\cform{k}(X) = \colim \CH^k(\XX_0)\] where the colimit ranges over all models of~$X$, with respect to the maps~$\pi_0^*$.
		
		\item \textit{The forms of type $(k,k)$ modulo the image of $\dz$ and $\dbz$}: \[\mform{k}(X) = \colim \CH_{n-k}(\XX_0)\] where the colimit ranges over all models of $X$, with respect to the maps~$\pi_0^!$.
		
		\item \textit{The closed currents of type $(k,k)$}: \[\ccurr{k}(X) = \lim \CH^k(\XX_0)\] where the limit ranges over all models of $X$, with respect to the maps~$\pi_{0!}$.
		
		\item \textit{The currents of type $(k,k)$ modulo the image of $\dz$ and $\dbz$}: \[\mcurr{k}(X) = \lim \CH_{n-k}(\XX_0)\] where the limit ranges over all models of $X$, with respect to the maps~$\pi_{0*}$.
	\end{description}
	
	If $T$ is a current in $\ccurr{}(X)$ or in $\mcurr{}(X)$, we will usually denote by $T_\XX$ its component on a model $\XX$.
	
	We have the following examples and constructions of forms and currents : 
	\begin{description}
		\item \textit{Inclusion maps} $[\, \cdot \, ] \colon \cform{}(X) \to \ccurr{}(X)$ and $\mform{}(X) \to \mcurr{}(X)$: they are defined in a way such that if $\om \in \cform{}(X)$ (resp. $\mform{}(X)$) is represented on a model~$\XX$ by $\om_\XX$, then for every model $\pi \colon \XX' \to \XX$ above~$\XX$, we have $[\om]_{\XX'} = \pi_{0}^* (\om_\XX)$ (resp.~$\pi_0^! (\om_\XX)$).
		
		\item \textit{Products}: if~$\om \in \cform{}(X)$ and $\eta$ is in one of the four groups, we can define a product $\om \wedge \eta$, of the same nature as $\eta$. 
		If $\eta \in \cform{}$ or~$\ccurr{}$, $\om \wedge \eta$ is represented by the product of the corresponding bivariant classes, otherwise it is represented by evaluating the classes defining~$\om$ at the cycles defining~$\eta$.
		
		\item \textit{Current of integration}: given~$Z \sbsq X$ a closed subvariety of codimension~$k$, the current of integration over~$Z$ is the element $\delta_Z \in \ccurr{k}$ defined by \[(\delta_Z)_\XX = i^* \eval{\XX}^{-1} \closure[\XX]{Z}\] where $\closure[\XX]{Z}$ is the closure of $Z$ in $\XX$ and $i$ is the inclusion $\XX_0 \into \XX$.
	\end{description}
	We also have analogs of these classical constructions:
	\begin{description}
		\item \textit{Differentiation maps} $\ddc \colon \mform{k}(X) \to \cform{k+1}(X)$ and $\mcurr{k}(X) \to \ccurr{k+1}(X)$: they are given on a model $\XX$ by the composition $i^* \circ \eval{\XX}^{-1} \circ i_*$, where $i$ is the inclusion of the special fiber of~$\XX_0$.
		
		\item \textit{Pullback of forms}: given $f \colon Y \to X$ a map of varieties over~$K$, let~$\XX$ be a model of~$X$ and $g \colon \YY \to \XX$ a model of~$Y$ over~$\XX$. If $\om \in \cform{k}(X)$ (resp.~$\mform{k}(X)$) is represented on~$\XX$ by~$\om_\XX$, we define $f^*\om \in \cform{k}(Y)$ (resp.~$\mform{k}(Y)$) represented on~$\YY$ by $g_0^*(\om_\XX)$ (resp.~$g_0^! (\om_\XX)$).
		
		\item \textit{Pushforward of currents}: in the same setup, given~$T$ a current on~$Y$, we define a current $f_* T$ on~$X$ by pushing forward the components of~$T$ using $g_{0*}$ or~$g_{0!}$ respectively.
	\end{description}
	
	We end this section by proving the projection formula, which is stated in \textcite[proposition~1.6.2] {bloch_gillet_soule--NonArchimedeanArakelovTheory1995} but whose proof was omitted, as well as a few compatibility lemmas.
		\begin{lm}\label{lm:bgs1}
		Let $g \colon \YY \to \XX$ be a proper map between $R$-schemes such that $g_!$ is well-defined.
		As above, denote by~$\eval{\XX}$ the evaluation map $\CH^\sbullet(\XX) \to \CH_\sbullet(\XX)$ given by $c \mapsto c(\Id_\XX, [\XX])$, and similarly for~$\eval{\YY}$. 
		Let also~$i$ and~$j$ be the inclusions $\XX_0 \into \XX$ and $\YY_0 \into \YY$, respectively. 
		Then the two following diagrams are commutative squares: 
		\[\begin{tikzcd}
			{\CH^\sbullet(\YY)} & {\CH^\sbullet(\XX)} \\
			{\CH^\sbullet(\YY_0)} & {\CH^\sbullet(\XX_0)}
			\arrow["{g_!}", from=1-1, to=1-2]
			\arrow["{j^*}", from=1-1, to=2-1]
			\arrow["{i^*}", from=1-2, to=2-2]
			\arrow["{g_{0!}}"', from=2-1, to=2-2]
		\end{tikzcd}
		\quad \text{and} \quad
		\begin{tikzcd}
			{\CH^\sbullet(\YY)} & {\CH^\sbullet(\XX)} \\
			{\CH_\sbullet(\YY)} & {\CH_\sbullet(\XX)}
			\arrow["{g_!}", from=1-1, to=1-2]
			\arrow["{\eval{\YY}}", from=1-1, to=2-1]
			\arrow["{\eval{\XX}}", from=1-2, to=2-2]
			\arrow["{g_*}"', from=2-1, to=2-2]
		\end{tikzcd}\]
	\end{lm}
	
	\begin{proof}
		Let $c \in \CH^\sbullet(\YY)$. 
		For the first diagram:
		\begin{align*}
			g_{0!}j^*(c) & = g_{0*}(j^*c \cdot [g_0]) \\ 
			& = g_{0*} (j^*c \cdot i^*[g]) && \text{\parencite[§1.3] {bloch_gillet_soule--NonArchimedeanArakelovTheory1995}}\\ 
			& = g_{0*}i^*(c \cdot [g]) && \text{\parencite[(A\textsubscript{13}) p. 323] {fulton--IntersectionTheory1998}} \\
			& = i^* g_* (c \cdot [g]) && \text{\parencite[(A\textsubscript{23})] {fulton--IntersectionTheory1998}}\\
			& = i^* g_!(c).
		\end{align*}
		
		For the second diagram,
		\begin{align*}
			\eval{\XX}(g_! c) & = g_!(c)([\XX]) 
			= g_*\bigl(c ([g][\XX])\bigr) 
			= g_*\bigl(c([\YY]) \bigr) 
			= g_*(\eval{\YY}(c))
		\end{align*}
		which ends the proof.
	\end{proof}

	\begin{lm}\label{lm:proj_chow}
		Let $g \colon \YY \to \XX$ as above, and let $c \in CH^\sbullet(\XX)$ and $d \in \CH^\sbullet(\YY)$. Then 
		\[g_! ( g^* c \proddot d) = c \proddot g_! (d)\]
	\end{lm}
	\begin{proof}
		Indeed using the definition of $g_!$ and \textcite[(A\textsubscript{123}) p. 323] {fulton--IntersectionTheory1998}, we have:
		\[g_! ( g^* c \proddot d) = g_* (g^* c \proddot d \proddot [g]) = c \proddot g_*(d \proddot [g]) = c \proddot g_!(d).\]
	\end{proof}
	
	\begin{prop}[Projection formula]\label{prop:projection_formula}
		Let~$f \colon Y \to X$ be a morphism between proper smooth $K$-schemes, $\om \in \cform{}(X)$ and $T \in \mcurr{}(Y)$ (resp. $\ccurr{}(Y)$). 
		Then 
		\[f_*(f^*(\om)\wedge T) = \om \wedge f_*(T).\]
	\end{prop}

	\begin{proof}
		Let~$\XX$ be a model of~$X$ on which $\om$ is represented by a class $c \in \CH^\sbullet(\XX_0)$, we only have to prove the equality of currents for models above~$\XX$. 
		Let $\XX' \to \XX$ be such a model, and let $\YY' \to \YY$ be models of~$Y$ over $\XX'$ and $\XX$ respectively, so that we have the following commutative diagram:
		\[\begin{tikzcd}
			{\YY'} & \YY \\
			{\XX'} & \XX
			\arrow["{\pi'}", from=1-1, to=1-2]
			\arrow["{g'}"', from=1-1, to=2-1]
			\arrow["g", from=1-2, to=2-2]
			\arrow["\pi", from=2-1, to=2-2]
		\end{tikzcd}\]
		
		Let us first do the case where $T \in \mcurr{}(Y)$, with $T_{\YY'} \in \CH_\sbullet(\YY'_0)$. 
		\begin{align*}
			\bigl(f_*\left(f^*\om \wedge T\right)\bigr)_{\XX'} & = g'_{0*} \bigl( (\pi'^*_0g^*_0c)(\Id, T_{\YY'}) \bigr)  \\
			& = g'_{0*} \bigl((g_0'^*\pi_0^*c)(\Id, T_{\YY'})\bigr) && \text{as } g_0 \circ \pi'_0 = \pi_0 \circ g'_0 \\
			& = g'_{0*}\bigl( (\pi_0^*c)(g'_0, T_{\YY'})\bigr) 
			&& \text{\parencite[(P\textsubscript{3}) p. 322] {fulton--IntersectionTheory1998}}\\ 
			& = (\pi_0^*c)(\Id, g'_{0*} T_{\YY'}) 
			&&\text{\parencite[(C\textsubscript{1}) p. 320] {fulton--IntersectionTheory1998}} \\ 
			& = \bigl( \om\wedge f_*T\bigr)_{\XX'}  
		\end{align*}
		
		Next, if $T \in \ccurr{}(Y)$, with $T_{\YY'} \in \CH^\sbullet(\YY'_0)$, then:
		\begin{align*}
			\bigl(f_*\left(f^*\om \wedge T\right)\bigr)_{\XX'} & = g'_{0!} \bigl(\pi'^*_0 g_0^* c \proddot T_{\YY'} \bigr) \\
			& = g'_{0!} \bigl(g_0'^*\pi^*_0 c \proddot T_{\YY'} \bigr) && \text{as } g_0 \circ \pi_0' = \pi_0 \circ g_0'\\ 
			& = \pi_0^* c \proddot g'_{0!} T_{\YY'} && \text{by lemma \ref{lm:proj_chow}}\\ 
			& = \bigl( \om\wedge f_*T\bigr)_{\XX'}
		\end{align*}
		The set of models~$\XX'$ above~$\XX$ being cofinal, this proves the desired result.
	\end{proof}
	
	\begin{lm}\label{lm:ddc_pushforward}
		Let $f \colon Y \to X$ be a morphism between smooth proper varieties over $K$, and $T \in \mcurr{}(Y)$. 
		Then $f_*(\ddc T) = \ddc (f_* T)$.
	\end{lm}
	
	\begin{proof}
		Let~$\XX$ be a model of~$X$, $g \colon \YY \to \XX$ a model of~$Y$ over~$\XX$. 
		We denote~$i$ and~$j$ the inclusions $\XX_0 \into \XX$ and $\YY_0 \into \YY$ respectively. 
		Then
		\begin{align*}
			\bigl(f_*(\ddc T)\bigr)_\XX & = g_{0!}j^* \eval{\YY}^{-1}j_* T_\YY \\ 
			& = i^*g_! \eval{\YY}^{-1}j_* T_\YY &&\text{by the first diagram in lemma~\ref{lm:bgs1}} \\
			& = i^* \eval{\XX}^{-1} g_* j_* T_\YY && \text{by the second diagram in the same lemma} \\
			& = i^* \eval{\XX}^{-1} i_* g_{0*} T_\YY &&\text{as }i\circ g = g_0\circ j \\
			& = \bigl(\ddc f_* T \bigr)_\XX
		\end{align*}
		as was to be shown.
	\end{proof}
	
	\begin{lm}\label{lm:compat_inclusion_wedge}		
		Let~$\om$,~$\eta \in \cform{}(X)$. Then $[\om\wedge\eta] = \om \wedge [\eta]$.
	\end{lm}
	
	\begin{proof}
		Let~$\XX$ be a model of~$X$ on which both $\om$ and $\eta$ admit representative $c$ and~$d$ in $\CH^\sbullet(\XX_0)$, such model exists by the hypothesis made in section~\ref{sec:models}.
		It is enough to show the equality on models of the form $\pi \colon \XX' \to \XX$, and for such models we have 
		\begin{align*}
			\bigl(\om \wedge [\eta]\bigr)_{\XX'} & = \pi_0^*(c) \proddot \pi_0^*(d) \\
			& = \pi_0^*(c \proddot d) 
			&& \text{\parencite[(A\textsubscript{13}) p. 323] {fulton--IntersectionTheory1998}} \\
			& = \bigl([\om \wedge \eta]\bigr)_{\XX'}
		\end{align*}
		This concludes the proof.
	\end{proof}
	
	\section{Metrics, Chern forms and the Poincar\'e-Lelong formula}\label{sec:poincaré-lelong}
	
	In this section, we recall and emphasize the discussion of \textcite{bloch_gillet_soule--NonArchimedeanArakelovTheory1995} on metrics on vector bundles. We then recall the non-archimedean equivalent of the function~$-\log\norm{s}^2$, where $s$ is a non-zero section of a given line bundle, and we prove in this setup an analog of the Poincaré-Lelong formula.
	
	\subsection{Metrics and Chern forms}
	
	\begin{defi}[\protect{\cite[1.9.1]{bloch_gillet_soule--NonArchimedeanArakelovTheory1995}}]\label{def:metric}		
		Let $F$ be a vector bundle on a proper smooth $K$-variety $X$. A \emph{metric} on~$F$ is a model $\mscr F$ of $F$ on $\lim \XX$, where the limit ranges over all models of $X$. In other words, a metric is the choice of a model $\XX$ and a vector bundle $\mscr F_\XX$ on $\XX$ with an isomorphism from the generic fiber of $\mscr F_\XX$ to $F$; and two such choices $\mscr F_\XX$ and $\mscr F_{\XX'}$ define the same metric if for any $\XX''$ mapping to both $\XX$ and $\XX'$, we have an isomorphism of models between the respective pullbacks.
		
		If a metric $\mscr F$ admits a representative $\mscr F_\XX$ on a model $\XX$, we will say that the model is \emph{compatible with the metric}. Note that the set of models compatible with a given metric is cofinal.

	\end{defi}
	
	\begin{defi}
		Given $F$ a vector bundle on $X$, $\mscr F$ a metric on $F$ and an integer~$k \geq 0$, we define the $k$-th \emph{Chern form} $c_k(\metr{F}) \in \cform{k}(X)$ by setting \[{(c_k(\metr{F}))_\XX = i^*(c_k(\mscr F_\XX) \cap \_ )}\] 
		on compatible models, where $i$ is the inclusion $\XX_0 \into \XX$.
	\end{defi}
	
	\medskip
	
	Let \[0 \to E \to F \to Q \to 0\] be an exact sequence of vector bundles on $X$, and let $\mscr F$ be a metric on $F$. We will define induced and quotient metric, following \textcite[1.9.7]{bloch_gillet_soule--NonArchimedeanArakelovTheory1995}.
		
		Let $\XX$ be a model of $X$ compatible with $\mscr F$. By the flattening theorems of \textcite[theorem 5.2.2]{raynaud_gruson--CriteresPlatitudeProjectivite1971}, there exists a model $\XX' \to \XX$ above~$\XX$, such that if we write $j$ the inclusion $X \to \XX'$, then the image of $\mscr F_{\XX'}$ into~$j_* Q$ is locally free.
		Denote by $\mscr Q_{\XX'}$ this image, and let $\mscr E_{\XX'}$ be the kernel of~$\mscr F_{\XX'} \to \mscr Q_{\XX'}$, it is also locally free.
	
	The vector bundle $\mscr Q_{\XX'}$ on $\XX'$ defines a metric on~$Q$, and one can see that this metric is independent of the choices of $\XX'$ and $\XX$. Indeed, this follows readily from the hypothesis made in~\ref{sec:models}, and the definition of $\mscr F$ being a metric.
	
	\begin{defi}\label{def:induced_metrics}
		The above metric on $Q$, defined on the model $\XX'$ by $\mscr Q_{\XX'}$, is called the \emph{quotient metric} on~$Q$. We define similarly the \emph{induced metric} on~$E$ using the vector bundle $\mscr E_{\XX'}$.
	\end{defi}

	\begin{lm}\label{lm:exact sq}
		Let \[0 \to E \to F \to Q \to 0\] an exact sequence of vector bundles. Endow $F$ with a metric, and endow $E$ and~$Q$ with the induced metric and the quotient metric respectively. Then for all~$k \geq 0$, we have the following equality in $\cform{k}(X)$:
		\[\ c_k(\metr{F}) = \sum_{i=0}^k c_i(\metr{E}) \wedge c_{k-i}(\metr{Q}).\]
	\end{lm}
	\begin{proof}
		Let $\XX$ be a model of $X$ compatible with all three metrics, and $i\colon \XX_0 \into \XX$ the inclusion of the special fiber. By definition of the induced and quotient metric, we have an exact sequence 
		\[0 \to \mscr E_\XX \to \mscr F_\XX \to \mscr Q_\XX \to 0\]
		and the Chern forms $c_k(\metr{F})$ (resp.\ $\metr{E}$, resp.\ $\metr{Q}$) are represented by the Chern classes $c_k(i^* \mscr F_\XX) \cap \_$ (resp.\ $\mscr E_\XX$, resp.\ $\mscr Q_\XX$) which are elements of the Chow group $\CH^k(\XX_0)$. But these do satisfy $c_\sbullet(i^*\mscr F_\XX) = c_\sbullet(i^*\mscr E_\XX) \, c_\sbullet(i^*\mscr Q_\XX)$ \parencite[theorem~2.3] {fulton--IntersectionTheory1998}, from which the lemma follows.
	\end{proof}
		
	\subsection{Divisors and the Poincaré-Lelong formula}
	\begin{defi}[Compare \protect{\cite[§3.1]{bloch_gillet_soule--NonArchimedeanArakelovTheory1995}}]\label{def:divnu}
		Let~$X$ be a proper smooth variety over $K$, of dimension~$n$. Let $L$ be a line bundle on~$X$, $\mscr L$ a metric~on $L$ and $s$ a regular section of~$L$ on~$X$. Given~$\XX$ a model of~$X$ compatible with the metric, we let $\t s$ be the extension of $s$ to a (possibly meromorphic) section of $\mscr L_\XX$, and we write 
		\begin{align*}
			\ddiv s \XX = \sum_V n_V [V] \in Z_n(\XX)
		\end{align*}
		for the divisor of~$\t s$. We also write
		\begin{align*}
			\ddiv[\nu] s \XX = \sum_{V \sbsq \XX_0} n_V [V] \in Z_n(\XX_0)
		\end{align*} for its vertical part, that is, the part supported on~$\XX_0$.
		
		In other words, we have $i_*\ddiv [\nu] s \XX = \ddiv s \XX - \closure[\XX]{\div s}$. 
		Here $\closure[\XX]{\placeholder}$ means the closure in $\XX$ of the corresponding cycle. 
		The collection $(\ddiv[\nu] s \XX)$ defines a current, which we denote by $\divnu s  \in \mcurr{0}(X)$.
	\end{defi}
	
	\noindent\textbf{Caution.} The current $\divnu{s}$ defined here is the opposite of the current $\div_\nu(s)$ defined in \textcite{bloch_gillet_soule--NonArchimedeanArakelovTheory1995}! 
	We chose  this convention here for the two following reasons: first, with this definition $\ddiv[\nu] s \XX$ is exactly the vertical part of the divisor of~$s$ on~$\XX$, as notation suggests. 
	Second, the current $\divnu s$ is then the analog in this setting of the current $[-\log \norm{s}^2]$ that appears in complex analysis, and the latter is arguably of greater significance than its opposite.
	
	\medskip 
	
	In \textcite[§3.1]{bloch_gillet_soule--NonArchimedeanArakelovTheory1995}, the Poincaré-Lelong formula is proven for a non-zero rational function. 
	We prove here the global version.
	
	\begin{prop}[Poincaré-Lelong formula]\label{prop:poincare_lelong}
		Let~$X$ be proper smooth, $L$ a line bundle on~$X$, $\mscr L$ a metric on~$L$ and $s$ a regular section of~$L$. 
		Then we have an equality of currents in $\ccurr{1}(X)$:
		\begin{align*}
			\ddc \divnu s + \delta_{\div s} = [c_1(\metr{L})].
		\end{align*}
	\end{prop}
	
	\begin{proof}
		It is enough to prove the equality of models of $X$ that are compatible with the metric: let~$\XX$ be such a model.
		By definition we have
		\begin{align*}
			\bigl( \ddc \divnu s + \delta_{\div s}\bigr)_\XX & = i^* \eval{\XX}^{-1}i_* \ddiv[\nu] s \XX + i^* \eval{\XX}^{-1} \closure[\XX]{\div s} \\ 
			& = i^* \eval{\XX}^{-1} \bigl( \ddiv s \XX - \closure[\XX]{\div s} + \closure[\XX]{\div s} \bigr) \\
			& = i^* \eval{\XX}^{-1} \ddiv s \XX \\ 
			\shortintertext{and} [c_1(\metr{L})]_\XX & = i^* (c_1(\mscr L_\XX) \cap \_).
		\end{align*}
		Therefore it is enough to show that \[\eval{\XX}^{-1} \ddiv s \XX = c_1(\mscr L_\XX) \cap \_\text \quad \text{ie} \quad c_1(\mscr L_\XX)\cap[\XX] = \ddiv s \XX\] which is in turn true by definition of the operator $c_1 (\mscr L_\XX)$ \parencite[§2.5]{fulton--IntersectionTheory1998}.
	\end{proof}
	
	\begin{rk}
		Let $p$ be a uniformizer of the discrete valuation ring $R$, and let $X$, $L$, $s$ and $\mscr L$ be as above. Then $s/p$ is another regular section of $L$ which defines the same cycle on~$X$, as $p$ is invertible in $K$. Furthermore, note that even though the currents $\divnu{s}$ and $\divnu{s/p}$ are not equal, their derivatives are, as $i_* \divnu{p} = \ddiv{p}{\XX} = 0$ in $\CH_\sbullet(\XX)$. Therefore, by multiplying the section~$s$ by an appropriate power of $p$, we can assume that the extension of $s$ on a compatible model $\XX$ is a global section of $\mscr L_\XX$, and not just a meromorphic one.
		
		This was not used in the proof above, but we will use this idea later when dealing with regular sections of vector bundles on~$X$.
	\end{rk}
	\section{Constructions of currents}\label{sec:construct}
	
	We present in this section two different constructions of currents in $\mcurr{0}$.
		
	First, we explain how, given two currents $S$ and $T$ in $\mcurr{0}(X)$, we can construct a new $(0,0)$-current $\inf(S,T)$. This construction is motivated by the need to define analogs of functions of the form $-\sum \log \norm{f_i}^2$, where the $f_i$'s are functions on a complex manifold, or sections of a line bundle.	
	Our general intuition on non-archimedean geometry tells us that the sum should be somehow replaced by a maximum; due to presence of the minus sign, this translates to a minimum in our setup.

	Secondly, we present some kind of gluing construction, and we explain how given an open covering of a model of $X$, and cycles on these open sets that are compatible between them and with respect to pushforward, we can put them together to form a current in $\mcurr{0}(X)$. 
	
	\medskip 
	
	Both of these constructions are not necessary for the statement and proof of theorem \ref{thm:bost-g-d-letter}, however they will be used in section \ref{sec:-log_sigma} to construct an explicit analog of the current $[-\log\norm{\sigma}^2]$ appearing in theorem \ref{thm:bost-g-s-complex}, where $\sigma$ is a section of a vector bundle.
	
	\subsection{Infimum of two \texorpdfstring{$(0,0)$-currents}{(0,0)-currents}} 
	
	Let~$\XX$ be a model of~$X$. There is a natural order on $Z_k(\XX_0)$, the set of $k$-cycles on $\XX_0$, given by $\sum n_V [V] \geq 0$ if and only if for all $V$, $n_V \geq 0$. With respect to this order, the infimum of two cycles $Z =  \sum_V n_V [V]$ and $Z' = \sum_V n'_V [V]$ exists, as it is given by 
	\[\inf(Z, Z') = \sum_V \min(n_V, n'_V) [V] \in Z_k(\XX_0).\]
	
	\begin{prop}\label{prop:min_courants}
		
		Let $\pi \colon \XX' \to \XX$ between models of $X$. Let $\alpha$ and $\beta$ be cycles in $Z_n(\XX'_0)$, where $n$ is the dimension of $X$. Then 
			\[\pi_{0*} \inf(\alpha, \beta) = \inf(\pi_{0*} \alpha, \pi_{0*}\beta).\]
	\end{prop}
	
	Before the proof, we first need the classical lemma: 
	\begin{lm}
		Let~$\pi \colon \YY \to \XX$ be birational and proper between integral and regular $R$-schemes, and $V$ a codimension 1 closed integral subscheme of~$\XX$. 
		Then there exists a unique codimension 1 closed integral subscheme $W$ of~$\YY$ such that $\pi(W) = V$, and such that $W$ satisfies $[k(W) \colon k(V)] = 1$.
	\end{lm}
	
	\begin{proof}[Proof (lemma)]
		Denote by~$U$ the biggest open subset of~$\XX$ such that the restriction $\pi \colon \pi^{-1}(U) \to U$ is an isomorphism. 
		Since $\pi$~is proper and birational, and $\XX$ is regular, the complement of $U$ has codimension at least~2 \parencite[\href{https://stacks.math.columbia.edu/tag/0BFP}{Tag 0BFP}]{stacks-project}.
		Now if~$V$ is a codimension 1 closed irreducible subscheme of~$\XX$, then its generic point~$\eta_V$ is in~$U$. 
		Since $\pi$ is an isomorphism over~$U$, the closure of the inverse image of~$\eta_V$ is the unique codimension 1 closed integral subscheme~$W$ above~$V$, and the equality $[k(W)\colon k(V)] = 1$ follows since $\pi$ induces an isomorphism of function fields.	
	\end{proof}
	\begin{proof}[Proof (proposition)]
		Let $\pi \colon \XX' \to \XX$, denote by~$i'$ and $i$ the inclusions of the respective special fibers. 
		Since $i_*$ is injective at the level of cycles, it is enough to prove 
		\begin{flalign*}
			&&i_* \pi_{0*} \inf(\alpha, \beta) & =  i_* \inf(\pi_{0*} \alpha, \pi_{0*}\beta), &&\\ 
			&\text{ie}& \pi_* \inf(i'_* \alpha, i'_* \beta) & = \inf(\pi_{*}i'_* \alpha, \pi_{*}i'_* \beta).&&
		\end{flalign*}
		
		Write $\pi_{*}i'_* \alpha = \sum_V a_V [V]$ and similarly for~$\beta$. 
		For all such~$V$, denote by~$W_V$ the codimension 1 subscheme of~$\XX'$ such that $\pi(W_V) = V$ and $[k(W_V):k(V)]=1$. Then $i'_*\alpha$ is of the form
		\[i'_*\alpha = \sum_V a_V [W_V] + \sum_Z a'_Z [Z]\] where the second sum is over varieties~$Z$ such that $\dim \pi(Z) < \dim Z$, and we have a similar decomposition for $i'_* \beta$.
		
		Then 
		\begin{align*}
			\pi_* \inf(i'_* \alpha, i'_* \beta) 
			& = \sum_V \min(a_V, b_V) \pi_* [W_V] + \sum_Z \min(a'_Z, b'_Z) \pi_* [Z] \\
			& = \sum_V \min(a_V, b_V)[V] \\ 
			& = \inf(\pi_{*}i'_* \alpha, \pi_{*}i'_* \beta)
		\end{align*}
		which concludes the proof.
	\end{proof}
	
	\begin{defi}\label{def:inf_courants}
		Denote by $n$ the dimension of $X$, and let $S$ and $T$ be currents in~$\mcurr{0}(X) = \lim \CH_n(\XX_0)$. As $\dim \XX_0 = n$, the quotient map $Z_n(\XX_0) \to \CH_n(\XX_0)$ is an isomorphism, therefore using proposition~\ref{prop:min_courants} we can define a new current $\inf(S,T)$ in $\mcurr{0}(X)$ by setting $\inf(S,T)_\XX = \inf(S_\XX, T_\XX)$.
	\end{defi}

	\begin{rk}
		The statement of proposition~\ref{prop:min_courants} and therefore the construction of definition~\ref{def:inf_courants} fail if the currents are of higher degree. 
		For a counter-example, take $X = \PP^1_K$, thus $n=1$. Let $\XX = \PP^1_R$, choose $x$ a closed point in the special fiber of $\XX$, and let $\XX'$ be the blow-up of $\XX$ at $x$. 
		Choose~$a$ and $b$ two distinct closed points in the exceptional divisor, let~$\alpha = [a]$ and~$\beta = [b]$, they are $0$-cycles.
		
		And 
		\begin{flalign*}
			&&\pi_{0*} \inf(\alpha, \beta) &= \pi_{0*} (0) = 0 &\\ 
			&\text{but}&\inf(\pi_{0*}\alpha, \pi_{0*}\beta) &= \inf([x], [x]) = [x].
		\end{flalign*}
	\end{rk}

	\begin{rk}
		A similar construction and proof allow us to define the supremum of two currents in~$\mcurr{0}(X)$. More generally, the proof of proposition~\ref{prop:min_courants} shows that we could define the infimum of any non-empty, bounded-below subset of~$\mcurr{0}(X)$, and similarly for the supremum. We will not use these constructions in the following.
	\end{rk}
	\subsection{Construction by gluing local data.}\label{sec:gluing}
	The spaces of differential forms and currents defined by \textcite{bloch_gillet_soule--NonArchimedeanArakelovTheory1995} do not have a sheaf structure. However, for the currents in $\mcurr{k}(X)$, it is possible to produce a construction rather similar to some kind of gluing construction. We state it here in the case $k=0$, as it is the only one we will need.
	
	\medskip
	
	We denote by $n$ the dimension of $X$, and we choose $\ZZZ$ a fixed model of $X$ as well as an open covering $\ZZZ = \bigcup_{a \in A} \mc U_a$. Suppose that for every model~$\rho \colon \XX \to \ZZZ$ above $\ZZZ$, and every $a\in A$, we are given an element
	\[\alpha_\XX^{(a)} \in Z_n(\XX_0 \cap \rho^{-1}(\mc U_a))\]  
	such that for all $a$ and $b$ in $A$: 
	\begin{equation}\label{eq:cycle_compat}
		(\alpha_\XX^{(a)})|_{\rho^{-1}(\mc U_{ab})} = (\alpha_\XX^{(b)})|_{\rho^{-1}(\mc U_{ab})}.
	\end{equation}
	
	Then for a fixed $\XX$ above $\ZZZ$, the collection of those cycles, indexed by $a$, glue to an~$n$-cycle~$\alpha_\XX \in Z_n(\XX_0)$.
	
	\begin{defi}\label{def:loc_compat_pushfw}
		A collection $(\alpha_\XX^{(a)})$ as above, satisfying (\ref{eq:cycle_compat}), is said to be \emph{locally compatible with pushforward} if for every map 
		\[\begin{tikzcd}[column sep=1em, row sep=2.5em]
			\XX' && \XX \\ 
			& \ZZZ &
			\arrow["\pi", from=1-1, to=1-3]
			\arrow["\rho", from=1-3, to=2-2]
			\arrow["\rho'"', from=1-1, to=2-2]
		\end{tikzcd}\]
		of models above $\ZZZ$, and for every $a\in A$, we have 
		\begin{equation}\label{eq:pushfwd_compat}
			(\pi_a)_{0*} \alpha_{\XX'}^{(a)} = \alpha_{\XX}^{(a)},
		\end{equation}
		where $\pi_a$ is the restriction $\rho'^{-1}(\mc U_a) \to \rho^{-1}(\mc U_a)$ of $\pi$.
	\end{defi}
	
	\begin{prop}\label{prop:courant_local}
		In the above situation, assume that the collection $(\alpha_\XX^{(a)})$ is locally compatible with pushforward. Then \[\pi_{0*} \alpha_{\XX'} = \alpha_\XX~;\]
		in other words, the collection $(\alpha_\XX)$ defines a current $\alpha \in \mcurr{0}(X)$.
		\medskip 
	\end{prop}
	\begin{rk}
		The current $\alpha$ defined in this way is independent of the choices made: given another model $\ZZZ'$ with covering $\ZZZ' = \bigcup_{i \in I} \mc V_i$, and another collection $\beta_\XX^{(i)}$ that is locally compatible with pushforward, if 
		\begin{equation}\label{eq:uniqueness}
			\alpha_\XX^{(a)}|_{\rho'^{-1}(\mc V_i)} = \beta_{\XX}^{(i)}|_{\rho^{-1}(\mc U_a)}
		\end{equation}
		is true for every model $\XX$ above both $\ZZZ$ and $\ZZZ'$, and every $a \in A$, $i \in I$, then the currents $\alpha$ and $\beta$ are equal.
		
		Indeed, the set of models $\XX$ above both $\ZZZ$ and $\ZZZ'$ is cofinal, and on such models, $\alpha_\XX = \beta_\XX$ immediately follows from~(\ref{eq:uniqueness}) and the fact that the open sets $\rho^{-1}(\mc U_a) \cap \rho'^{-1}(\mc V_i)$ cover $\XX$.
	\end{rk}
	\begin{proof}[Proof of proposition \ref{prop:courant_local}]
		Let $\iota_a$ (resp.\ $\iota'_a$) be the inclusion~$\rho^{-1}(\mc U_a) \to \XX$ (resp.\ $\rho'^{-1}(\mc U_a) \to \XX'$).
		Then~(\ref{eq:pushfwd_compat}) reads
		\[(\pi_a)_{0*} (\iota'_a)_0^* \alpha_{\XX'} = (\iota_a)_0^* \alpha_\XX \]
		and the first term may be rewritten as $(\iota_a)_0^* \pi_{0*} \alpha_{\XX'}$ \parencite[proposition~1.7] {fulton--IntersectionTheory1998}, so that for all $a \in A$, 
		\begin{align*}
			(\iota_a)_0^* \left(\pi_{0*} \alpha_{\XX'} - \alpha_\XX \right) & = 0
			\shortintertext{that is}
			\left(\pi_{0*} \alpha_{\XX'} - \alpha_\XX \right)|_{\rho^{-1}(\mc U_a)\cap \XX_0} & = 0.
		\end{align*}
		
		Since cycles are involved, and since the $\rho^{-1}(\mc U_a)\cap \XX_0$'s cover $\XX_0$, it follows that $\pi_{0*} \alpha_{\XX'} = \alpha_\XX$, as we wanted to show.
	\end{proof}

		\begin{rk}
		Because of the canonical isomorphism between $Z_n(\XX_0)$ and $\CH_n(\XX_0)$, this proposition looks quite similar to the statement that $\mcurr{0}(X)$ is a sheaf. However this interpretation is not exactly true, as we had to take the open cover on a model $\ZZZ$ and not directly on $X$. 
		
		Actually, given a current $\alpha \in \mcurr{0}(X)$ and an open set $U \sbsq X$, we see no way of defining the restriction $\alpha|_U$ in a natural way.

		Note also that a similar statement holds, using the same proof, if we instead choose $\alpha_\XX^{(a)}$ to be a $(n-k)$-cycle. The current $\alpha$ thus obtained  obtained is then an element of $\mcurr{k}(X)$. But since $Z_k(\XX_0) \to \CH_k(\XX_0)$ is not an isomorphism anymore when $k \geq 1$, the interpretation in terms of ``sheaf gluing'' strays even further. We will only use the case $k=0$ in the following.
	
	\end{rk}
	\section{Bost-Gillet-Soulé formula: statement and proof}
	
	We prove in this section the analog of \textcite[§1.2.3]{bost_gillet_soule--HeightsProjectiveVarietiesPositive1994} in our non-archimedean setup. We will use the following convention: given $F$ a vector bundle on $X$, we denote by~$\PP(F) = \Proj (\Sym F^\vee)$ its projectivization and by~$\OO_F(1)$ the twisting sheaf on $\PP(F)$.
	
	\subsection{Setup and statement} 
	Let $X$ be a proper smooth variety over $K$, of pure dimension $n$, and $Z$ a proper closed subset of $X$. We will assume until the end that there exists a vector bundle $F$ of rank $r$ on $X$, and a global section $\sigma$ of $F$, which is \emph{regular}, and such that $Z$ is the zero-set of $\sigma$. Note that by the regularity hypothesis, $Z$ has codimension $r$ everywhere in $X$.
	
	Let $\mc I$ be the ideal sheaf in $\OO_X$ defining $Z$, $\pi \colon \t X \to X$ the blow-up of $X$ along $Z$, and $E$ the exceptional divisor.
	
	The global section $\sigma$ induces a surjective morphism $F^\vee \onto \mc I$, hence a surjection
	\[\Sym (F^\vee) \onto \Sym(\mc I) = \bigoplus_k \mc I^k\] 
	and thus a closed immersion $f \colon \t X \into \PP(F)$. This immersion satisfies
	\begin{equation}\label{eq:O_F(1)}
		f^*\OO_F(1) = \mc I \proddot \OO_{\t X} = \OO_{\t X}(-E).
	\end{equation}
	
	Now let us choose a metric $\mscr F$ on $F$, in the sense of definition~\ref{def:metric}. Then using definition~\ref{def:induced_metrics} we endow $\OO_F(-1)$ and $Q_F := p^*F / \OO_F(-1)$ with the induced metric and the quotient metric, respectively. In view of~(\ref{eq:O_F(1)}), we also endow~$\OO_{\t X}(E)$ with a metric by pulling back the metric from $\OO_F(-1)$.
	
	\begin{thm}\label{thm:bost-g-s}
		In the situation as above, denote by $s$ the canonical section of~$\OO_{\t X}(E)$.
		
		We define a current $\Lambda \in \mcurr{r-1}(X)$ by 
		\[\Lambda = \pi_* (\divnu{s} \wedge f^* c_{r-1}(\metr{Q_F})).\]
		Then we have the following equality of currents:
		\[\ddc \Lambda + \delta_Z = [c_r(\metr{F})].\]In particular $\Lambda$ is a Green current for $Z$.
	\end{thm}
	
	\subsection{Blow-up and models}\label{sec:blow-up_models}
	Before proving the theorem in the next part, we will first explain in this part how to construct models $\t \XX$ of $\t X$ that will suit our needs. 
	
	Let us fix a model $\XX$ of $X$ compatible with the metric. Possibly after multiplying $\sigma$ by an appropriate power of the uniformizer of $R$, the section $\sigma$ extends to a global section $\t \sigma$ of $\mscr F_\XX$ (see the remark after proposition~\ref{prop:poincare_lelong}), and doing so is harmless as it does not change the closed subset $Z$ of~$X$. 
	
	Choose an open covering $\XX = \bigcup_{a \in A} \mc U_a$ which trivializes $\mscr F_\XX$. For every $a$ in $A$, let~$(\eps^{(a)}_k)_{1 \leq k \leq r}$ be a local frame of~$\mscr F_\XX$ over~$\mc U_a$, then the extension $\t \sigma$ decomposes locally as $\t \sigma |_{\mc U_a} = \sum f_k^{(a)} \eps_k^{(a)}$, with $f_k^{(a)} \in \OO_\XX(\mc U_a)$.

	\begin{defi}\label{def:model_blowup}
		Let $\mc J$ be the sheaf ideal in $\OO_\XX$ locally defined on $\mc U_a$ by
		\[\mc J_{|\mc U_a} = (f_1^{(a)}, \dots, f_r^{(a)}).\]
		
		We define $\t \XX = \Bl_{\mc J} \XX$, and we write $u$ for the blow-up map.
	\end{defi}
	
	One first thing to note is that $\t \XX$ is a model of $\t X$. Indeed, by definition, the ideal $\mc I_Z$ of $Z$ inside $X$ is locally defined by the images of the $f_i^{(a)}$ under the map $\OO_\XX \mapsto \OO_\XX \otimes_R K \cong \OO_X$, therefore $\t \XX \otimes_R K \cong \t X$.
	
	Furthermore, as was the case on $X$, the section $\t \sigma$ on $\XX$ induces a surjection $\Sym(\mscr F_\XX^\vee) \onto \bigoplus \mc J^n$, and therefore a closed immersion $g \colon \t \XX \into \PP(\mscr F_\XX)$. We write~$q$ for the projection $\PP(\mscr F_\XX) \to \XX$, and~$i$, $j$ and $j'$ for the respective inclusion of special fibers in $\XX$, $\t\XX$ and $\PP(\mscr F_\XX)$. 
	
	We summarize the notation about these models in the following commutative diagrams (the map $q_0$ is omitted for readability reasons only):
	\begin{equation}\label{eq:diagram}
		\begin{tikzcd}
	& {\PP(F)} &&& {\PP(\mscr F_\XX)}_0 && {\PP(\mscr F_\XX)} \\
	{\t X} &&& {\t \XX_0} && {\t\XX} \\
	& X &&& {\XX_0} && \XX \\
	\arrow["i", hook, from=3-5, to=3-7]
	\arrow["j", hook, from=2-4, to=2-6]
	\arrow["j'", hook, from=1-5, to=1-7]
	\arrow["\pi", from=2-1, to=3-2]
	\arrow["f"', hook, from=2-1, to=1-2]
	\arrow["p", from=1-2, to=3-2]
	\arrow["u", from=2-6, to=3-7]
	\arrow["g"', hook, from=2-6, to=1-7]
	\arrow["q", from=1-7, to=3-7]
	\arrow["u_0", from=2-4, to=3-5]
	\arrow["g_0"', hook, from=2-4, to=1-5]
	\end{tikzcd}
	\end{equation}

	\begin{rk}
		Even though the sequence defining $\sigma$ on $X$ is regular, there is no reason for the sequence $(f_i^{(a)})_{1 \leq i \leq r}$ to be regular on the model $\XX$. 
	
	Furthermore, note that we only have the inclusion $\mc J \sbsq \mc I_{\closure{Z}} := u^{-1}(\mc I_Z)$, and in general the closed subset defined by $\mc J$ will be bigger than $\closure{Z}$, as it may have other irreducible components in the special fiber.
	\end{rk}
	
	\subsection{Proof of theorem~\ref{thm:bost-g-s}}
	We now move on to the proof of our analog of Bost-Gillet-Soulé formula. First, using lemma~\ref{lm:ddc_pushforward} and the Poincaré-Lelong equation (proposition~\ref{prop:poincare_lelong}) we obtain
		\begin{align}
			\ddc \Lambda & = \ddc \pi_* \left( \divnu{s} \wedge f^* c_{r-1}(\metr{Q_F}) \right)  
			 = \pi_* \left( (\ddc \divnu{s})\wedge f^*c_{r-1}(\metr{Q_F})\right) \nonumber \\ 
			& = \pi_* \left( ([c_1(\metr{\OO_{\t X}(E)})] - \delta_E)\wedge f^*c_{r-1}(\metr{Q_F})\right). \label{eq:thm0}
		\end{align}
		
		First, recall from~(\ref{eq:O_F(1)}) that $\OO_{\t X}(E) = f^* \OO_F(-1)$, and that the metric on~$\OO_{\t X}(E)$ was defined to be the pullback of the metric on $\OO_F(-1)$. Therefore using lemma~\ref{lm:compat_inclusion_wedge}:
		
		\begin{align*}
			[c_1(\metr{\OO_{\t X}(E)})] \wedge f^*c_{r-1}(\metr{Q_F}) & = [f^*c_1(\metr{\OO_F(-1)}) \wedge f^* c_{r-1}(\metr{Q_F})].
		\end{align*}
		Using lemma~\ref{lm:exact sq} applied to the exact sequence 
		\begin{align}\label{eq:thm1}
			0 \to f^* \OO_F(-1) \to \pi^* F \to f^* &Q_F \to 0, \nonumber \\
			\shortintertext{we get}
			[c_1(\metr{\OO_{\t X}(E)})] \wedge f^*c_{r-1}(\metr{Q_F}) & = [\pi^* c_r(\metr{F})] \nonumber \\
			\shortintertext{and thus}
			\pi_* \left( [c_1(\metr{\OO_{\t X}(E)})] \wedge f^*c_{r-1}(\metr{Q_F}) \right)& = [c_r(\metr{F})]
		\end{align}
		by the projection formula~\ref{prop:projection_formula}.

		\medskip 
		
		Now we compute the second term $\pi_*(\delta_E \wedge f^*c_{r-1}(\metr{Q_F}))$ appearing in~(\ref{eq:thm0}). We start by noting that using lemma~\ref{lm:exact sq}, we have 
		\[c_{\sbullet}(\metr{Q_F}) = c_\sbullet(p^* \metr{F}) \proddot (1 - c_1(\metr{\OO_F(1)}))^{-1} = c_\sbullet(p^* \metr{F}) \sum_{k \geq 0} c_1(\metr{\OO_F(1)})^k.\]This implies the following decomposition of the Chern forms of $Q_F$: for all $l \geq 0$,
		\begin{equation}\label{eq:c_l(Q)}
		\ c_l(\metr{Q_F}) = \sum_{k=0}^l c_{l-k}(p^*\metr{F})\wedge c_1(\metr{\OO_F(1)})^k.
		\end{equation}
		
		Applying with $l=r-1$ gives
		\begin{align}\label{eq:thm2}
			\pi_*(\delta_E \wedge f^*c_{r-1}(\metr{Q_F}) ) 
			& = \pi_* \left( \delta_E \wedge \sum_{k=0}^{r-1} \pi^* c_{r-1-k}(\metr{F}) \wedge f^*c_1(\metr{\OO_F(1)})^k \right) \nonumber\\ 
			& = \sum_{k=0}^{r-1} c_{r-1-k}(\metr{F}) \wedge \pi_* \left( \delta_E \wedge f^*c_1(\metr{\OO_F(1)})^k \right)
		\end{align}
		by the projection formula. Let $0 \leq k \leq r-1$, and $\XX$ a model of $X$. We will compute $\pi_*(\delta_E \wedge f^*c_1(\metr{\OO_F(1)})^k)$ on models, using the model $\t \XX$ of $\t X$ defined in~\ref{def:model_blowup}, and using notation as in diagram~(\ref{eq:diagram}).
		\begin{align*}
			\Big( \pi_*\big(\delta_E \wedge f^*c_1(\metr{\OO_F(1)})^k\big)\Big)_\XX
			& = u_{0!}\Big(j^* \eval{\t \XX}^{-1}(\closure{E}) \proddot g_0^* c_1\big((j')^* \OO_{\mscr F_\XX}(1)\big)^k\Big) \\ 
			& = i^* u_! \Big( \eval{\t \XX}^{-1}(\closure{E}) \proddot g^* c_1\big( \OO_{\mscr F_\XX}(1))^k\big)\Big)
		\end{align*}
		by \textcite[(A\textsubscript{13}) p. 323] {fulton--IntersectionTheory1998} and lemma~\ref{lm:bgs1}. And using commutativity of bivariant Chow rings (§\ref{sec:chow}), as well as lemmas~\ref{lm:proj_chow} and \ref{lm:bgs1}, we obtain
		\begin{align*}
			\Big( \pi_*\big(\delta_E \wedge f^*c_1(\metr{\OO_F(1)})^k\big)\Big)_\XX
			& = i^* q_! g_! \Big( g^* c_1\big( \OO_{\mscr F_\XX}(1)\big)^k \proddot \eval{\t \XX}^{-1}(\closure{E}) \Big)\\ 
			& = i^* q_! \Big( c_1\big( \OO_{\mscr F_\XX}(1)\big)^k \proddot \eval{\PP(\mscr F_\XX)}^{-1}(g_* \closure{E})\Big).
		\end{align*}

	We now make use of the following lemma, whose proof crucially uses the hypothesis of regularity of the section $\sigma$:
	\begin{lm}\label{lm:closure}
		In the above situation, we have $g_* \closure{E} = q^* \closure{Z}$.
	\end{lm}
	Let us first finish proving the theorem. By the above lemma, we have
	\begin{align*}
		\eval{\XX} \Big( q_! \big( c_1( \OO_{\mscr F_\XX}(1))^k \proddot \eval{\PP(\mscr F_\XX)}^{-1}(g_* \closure{E}) \Big)
		& = q_* \Big( c_1( \OO_{\mscr F_\XX}(1))^k \cap q^* \closure{Z} \Big)\\
		& = s_{k-(r-1)}(\mscr F_\XX) \cap \closure{Z}.
	\end{align*}
	Here $s_{k-(r-1)}(\mscr F_\XX)$ is the $(k-(r-1))$-Segre class of $\mscr F_\XX$. But $k-(r-1) \leq 0$, therefore from \textcite[proposition~3.1.(a)] {fulton--IntersectionTheory1998} follows
	\[s_{k-(r-1)}(\mscr F_\XX)\cap \closure{Z} = 
	\begin{cases}
		\closure{Z} & \text{if $k=r-1$}\\
		0 & \text{otherwise.}
	\end{cases}\]
	Hence, 
	\begin{equation}
		\Big(\pi_*(\delta_E \wedge f^*c_1(\metr{\OO_F(1)})^{r-1}) \Big)_\XX = 
		\begin{cases}
			i^* \eval{\XX}^{-1} \closure{Z} = (\delta_Y)_\XX & \text{if $k=r-1$} \\ 
			0 & \text{otherwise.}
		\end{cases}
	\end{equation}
	Consequently, (\ref{eq:thm2}) now becomes $\pi_*(\delta_E \wedge f^* c_{r-1}(\metr{Q_F})) = \delta_Y$, and putting this together with~(\ref{eq:thm1}) inside~(\ref{eq:thm0}), we finally obtain
	\[\ddc \Lambda = [c_r(\metr{F})] - \delta_Y.\]
	\qed

	We have now completed the proof of theorem~\ref{thm:bost-g-s}, and all that remains is to prove lemma~\ref{lm:closure}.
	\begin{proof}[Proof of lemma]
		We will actually prove that the image of $\closure{E}$ under the inclusion $\t \XX \into \PP(\mscr F_\XX)$ is exactly $q^{-1}(\closure{Z})$. 
		That is, if we write $\mc I_{\closure{Z}}$ and $\mc I_{\closure{E}}$ for the ideal sheaves of $\closure{Z}$ and $\closure{E}$ respectively, we will show that the surjection $\psi \colon \Sym(\mscr F_\XX^\vee) \onto \OO_{\t \XX}$ described in section~\ref{sec:blow-up_models} induces an isomorphism
		\[\quotient{\Sym(\mscr F_\XX^\vee)}{\mc I_{\closure{Z}}\Sym(\mscr F_\XX^\vee)} 
		\isom 
		\quotient{\OO_{\t \XX}}{\mc I_{\closure{E}}}.\]
		
		We do this locally, so we may assume that $\XX = \Spec A$, that $\mscr F_\XX^\vee$ trivializes as~$\bigoplus_{i=1}^{r} A \proddot x_i$ and that the extension of~$\sigma$ to $\XX$ decomposes as $(f_1, \dots, f_r)$, with~$f_i \in A$. Then $X = \Spec (A \otimes K)$, and the open immersion $X \to \XX$ is induced by the canonical map $w \colon A \to A\otimes K$. 
		Furthermore we have a canonical identification $\Sym(\mscr F_\XX^\vee) \simeq A[x_1, \dots x_r]$.
		The situation is summarized in the following commutative diagram:
			\[\begin{tikzcd}
			{A[x_1, \dots, x_r]} & \OO_{\t \XX} \\
			{(A\otimes_R K)[x_1, \dots, x_r]} & {\OO_{\t\XX} \otimes_R K = \OO_{\t X}} 
			\arrow["\psi", two heads, from=1-1, to=1-2]
			\arrow["w", from=1-1, to=2-1]
			\arrow["w", from=1-2, to=2-2]
			\arrow["\phi", two heads, from=2-1, to=2-2]
		\end{tikzcd}\]
		The vertical arrows are induced by $w \colon A \to A\otimes K$ and are still denoted by $w$.
		
		\medskip 
		
		In this situation, the ideal defining $Z$ is $I_Z = (w(f_1), \dots, w(f_r)) \sbsq A\otimes K$, and because this sequence is regular, the ideal defining $E$ is $I_E = I_Z \proddot \OO_{\t X}$ \parencite[lemma~A.6.1] {fulton--IntersectionTheory1998}. Moreover, again by regularity of the sequence defining $Z$, the kernel of the bottom map $\phi$ is exactly the ideal generated by the $x_i w(f_j) - x_j w(f_i)$, for $1 \leq i,j \leq r$.
		
		\medskip 
		The ideal $I_{\closure{Z}}$ of $\closure{Z}$ in $\XX$ is defined to be $w^{-1}(I_Z) \sbsq A$, and similarly, $I_{\closure{E}} = w^{-1}(I_E) \sbsq \OO_{\t \XX}$. Therefore we want to show that the top map $\psi$ of the previous diagram induces an isomorphism 
		\[\overline{\psi} \colon A/I_{\closure{Z}}[x_1, \dots, x_r] \isom \OO_{\t \XX}/I_{\closure{E}}.\]
		
		First, if $P$ is an element of $I_{\closure{Z}}[{x_1, \dots, x_r}]$, then $w(P) \in I_Z[x_1, \dots, x_r]$, hence~$w(\psi(P)) = \phi(w(P))$ is in $I_Z \proddot \OO_{\t X} = I_E$, that is, $\psi(P) \in I_{\closure{E}}$. This shows that $\overline \psi$ is well-defined.
		
		\medskip
		
		The surjectivity of $\overline \psi$ is ensured because $\psi$ is surjective; so let us show the injectivity of $\overline \psi$. Take $P \in A[x_1, \dots, x_r]$ and suppose $\psi(P) \in I_{\closure{E}}$. Then $w(\psi(P))$ is in $I_E = I_Z \proddot \OO_{\t X}$, therefore by surjectivity of $\phi$ we can choose $Q \in I_Z[x_1, \dots, x_r]$ such that $\phi(Q) = w(\psi(P))$.
		
		But then $w(P)-Q$ is in the kernel of $\phi$, and this kernel being generated by the $(x_i w(f_j)-x_j w(f_i))$, it is included in $I_Z[x_1, \dots, x_r]$. Therefore $w(P) \in I_Z[x_1, \dots, x_r]$, ie $P \in I_{\closure{Z}}[x_1, \dots, x_r]$, thus proving the injectivity.
		
		This concludes the proof of the lemma, and hence completes the proof of theorem~\ref{thm:bost-g-s}.

	\end{proof}

	\section{An explicit analog of \texorpdfstring{$[-\log \norm{\sigma}^2]$}{[-log ||sigma||2]}}\label{sec:-log_sigma}
	
	Now that theorem~\ref{thm:bost-g-s} is proved, let us investigate one difference between its statement and the original complex geometry statement (theorem~\ref{thm:bost-g-s-complex}). In theorem~\ref{thm:bost-g-s} we made use of $s$, the canonical section of $\OO_{\t X}(E)$, where $E$ is the exceptional divisor of the blow-up $\pi \colon \t X \to X$ along~$Z$. Whereas in the complex statement, the corresponding term is replaced by $\pi^* \log \norm{\sigma}^2$.

	This is justified by the following fact : in a complex setup, we have $\pi^* \sigma = s$ \parencite[§1.2.3]
	{bost_gillet_soule--HeightsProjectiveVarietiesPositive1994}. This implies in particular that 
	\[-\log \norm{s}^2 = -\pi^* \log \norm{\sigma}^2\]
	justifying its use in theorem~\ref{thm:bost-g-s}. Note that this also implies the equality of currents on $X$:
	\begin{equation}\label{eq:s_sigma_courants}
		\pi_* [-\log \norm{s}^2] = [-\log \norm{\sigma}^2].
	\end{equation}
	
	In the non-archimedean setup of Bloch-Gillet-Soulé, the analog of $[-\log \norm{s}^2]$ is given by the current $\divnu{s}$, as explained in section~\ref{sec:poincaré-lelong}, but there is no analog yet of~$[-\log \norm{\sigma}^2]$ when $\sigma$ is a section of a vector bundle of rank $r \geq 2$.
	
	In this section we construct a current $\divnu{\sigma} \in \mcurr{0}(X)$ using the constructions of section~\ref{sec:construct}, and we show that $\pi_* \divnu{s} = \divnu{\sigma}$, justifying in view of~(\ref{eq:s_sigma_courants}) the fact that $\divnu{\sigma}$ is the right analog of $[-\log \norm{\sigma}^2]$.
	
	If $\divnu{\sigma}$ were a form, then the current $\Lambda$ appearing in theorem~\ref{thm:bost-g-s} could be rewritten as 
	\[\pi_* [\pi^* \divnu{\sigma} \wedge f^* c_{r-1}(\metr{Q_F})],\]
	in good analogy with theorem~\ref{thm:bost-g-s-complex}. However this will not be the case in general, as was already noted in section~\ref{sec:poincaré-lelong} for sections of line bundles.
	\bigskip

	To construct this current $\divnu{\sigma}$, as in section~\ref{sec:blow-up_models}, let us fix a model $\XX$ of $X$ compatible with the metric, and choose an open covering $\XX = \bigcup_{a \in A} \mc U_a$ which trivializes $\mscr F_\XX$. 
	For every $a$ in $A$, let~$(\eps^{(a)}_k)_{1 \leq k \leq r}$ be a local frame of~$\mscr F_\XX$ over~$\mc U_a$. 
	We denote by~$h^{(ab)} \in \GL_r(\OO_{\mc U_{ab}})$ the induced transition function of~$\mscr F_\XX$ from~$\mc U_a$ to~$\mc U_b$, by this we mean 
	\[(\eps_1^{(b)}, \dots, \eps_r^{(b)}) = (\eps_1^{(a)}, \dots, \eps_r^{(a)}) h^{(ab)} \qquad \text{on $\mc U_a \cap \mc U_b$.}\]

	For all $k \leq r$ and $a\in A$, we write $\sigma |_{\mc U_a} = \sum f_k^{(a)} \eps_k^{(a)}$, with $f_k^{(a)} \in \OO_\XX(\mc U_a)$. Then by definition we have $f^{(a)} = h^{(ab)} f^{(b)}$ on~$U_a \cap U_b$ (where we wrote $f^{(a)}$ for the transpose of $(f_1^{(a)}, \dots, f_r^{(a)})$).
	
	\begin{lm}\label{lm:min_ord_f_ia}
		For any $a,b \in A$, and for all codimension 1 closed irreducible subscheme $V$ of $\XX$ intersecting both $\mc U_a$ and $\mc U_b$, we have 
	\[\min (\ord_V(f_1^{(a)}), \dots, \ord_V(f_r^{(a)})) = \min (\ord_V(f_1^{(b)}), \dots, \ord_V(f_r^{(b)})).\]
	\end{lm}
	
	\begin{proof}
		For all $j$, we have
	\begin{gather*}
		\begin{align*}
			\ord_V(f_j^{(a)}) & = \ord_V\left(\sum_{k=1}^r h^{(ab)}_{jk} f_k^{(b)}\right) 
		 \geq  \min_k \left( \ord_V(h^{(ab)}_{jk} f_k^{(b)})\right)  \geq \min_k \left(\ord_V(f_k^{(b)}) \right)
		\end{align*}
		\shortintertext{therefore}
		\min_k \left(\ord_V(f_k^{(a)}) \right)  \geq \min_k \left(\ord_V(f_k^{(b)}) \right)
	\end{gather*}
	and the other inequality follows from symmetry.
	\end{proof}
	
	The same holds virtually without any change for any model $\rho \colon \XX' \to \XX$ above $\XX$ : we do the same computations, replacing the $f_i^{(a)}$ and the $h^{(ab)}$ by their respective pullback by $\rho$. 
	
	\medskip
	Therefore, the previous lemma show that the collection 
		\[\inf\left( \divnu{f_1^{(a)}}_{\rho^{-1}(\mc U_a)}, \dots, \divnu{f_r^{(a)}}_{\rho^{-1}(\mc U_a)} \right) \in Z_n(\rho^{-1}(\mc U_a) \cap \XX'_0),\]
	satisfies condition (\ref{eq:cycle_compat}) from section~\ref{sec:gluing}, and proposition~\ref{prop:min_courants} shows that it is locally compatible with pushforward, in the sense of definition~\ref{def:loc_compat_pushfw}. 
	
	\begin{defi}\label{def:div_nu_sigma}
	We denote by $\divnu{\sigma}$ the current in $\mcurr{0}(X)$ obtained by applying the gluing construction of proposition~\ref{prop:courant_local} to the above collection of cycles.
	\end{defi}

	\begin{rk}
		This current is indeed independent of the choices made: if $\ZZZ = \bigcup_{i \in I} \mc V_i$ is another model with open cover trivializing $\mscr F_\ZZZ$, then for every model $\XX'$ above both $\XX$ and $\ZZZ$, the vector bundle $\mscr F_{\XX'}$ is trivialized in two different ways by pulling back the trivializations from $\XX$ and $\ZZZ$ respectively. Thus there are transition functions to go from one to the other, and the same computation as before, together with the remark after proposition~\ref{prop:courant_local}, show that the two induced currents coincide.
	\end{rk}
	\begin{rk}
		If $F$ is actually a line bundle, then the current $\divnu{\sigma}$ defined just above coincide with the one previously defined in~\ref{def:divnu}.
	\end{rk}
	\begin{prop}\label{prop:divnusigma_bien_def}
		As above, denote by $E$ the exceptional divisor of $\t X$ and $s$ the canonical section of $\OO_{\t X}(E)$. We have the following equality of currents in~$\mcurr{0}(X)$:
		\[\pi_* \divnu{s} = \divnu{\sigma}.\]
	\end{prop}
	
	\begin{proof}	
		We will keep the notation used in the discussion before definition~\ref{def:div_nu_sigma}: we have a fixed model $\XX$, compatible with the metric, an open covering $(\mc U_a)$ of~$\XX$ on which $\mscr F_\XX$ trivializes, and locally we write $\sigma|_{\mc U_a} = \sum f_i^{(a)} \eps_i^{(a)}$.
		
		It is enough to show the equality of currents on models $\XX'$ above $\XX$, as the set of such models is cofinal. Replacing the $f_i^{(a)}$ by their pullback and the $\mc U_a$ by their inverse image, we can assume that $\XX' = \XX$, that is we only need to show 
		\[(\pi_* \divnu{s})_\XX = \ddiv[\nu]{\sigma}{\XX} \qquad \text{in }Z_n(\XX_0).\]
		
		Denote by $\iota_a$ the inclusion $\mc U_a \to \XX$. Since cycles are involved, and since the~$\mc U_a\cap \XX_0$'s cover $\XX_0$, it is enough to show that for all $a \in A$, 
		\[(\iota_a)_0^* (\pi_* \divnu{s})_\XX = (\iota_a)_0^* \ddiv[\nu]{\sigma}{\XX}.\]
		Note that by definition, 
		\[(\iota_a)_0^* \ddiv[\nu]{\sigma}{\XX} = \inf \left( \ddiv[\nu]{f_1^{(a)}}{\mc U_a}, \dots, \ddiv[\nu]{f_r^{(a)}}{\mc U_a}\right).\]
		
		Let $u \colon \t \XX \to \XX$ be the model of $\t X$ defined in~\ref{def:model_blowup}: it is obtained by blowing-up the closed subscheme of $\XX$ locally defined by $\{f_1^{(a)} = \dots = f_r^{(a)} = 0\}$. 
		Write~$\mc V_a = u^{-1}(\mc U_a)$, then we have the following commutative diagram:
		\[\begin{tikzcd}
			\mc V_a = {u^{-1}(\mc U_a)} & {\t\XX} \\
			{\mc U_a} & \XX
			\arrow["{\iota'_a}", from=1-1, to=1-2]
			\arrow["{u_a}"', from=1-1, to=2-1]
			\arrow["u", from=1-2, to=2-2]
			\arrow["{\iota_a}"', from=2-1, to=2-2]
		\end{tikzcd}\]

		As seen in section~\ref{sec:blow-up_models}, there is a canonical closed immersion $g \colon \t\XX \into \PP(\mscr F_\XX)$. Since blowing-up is local, we can look at this immersion locally: the induced map is
		\[\Sym(\mscr F_\XX)|_{\mc U_a}) \simeq \OO_{\mc U_a} [x_1, \dots, x_r] \onto \bigoplus_n (f_1^{(a)}, \dots, f_r^{(a)})^n ;\]
		therefore $g$ identifies $\mc V_a$ with a closed subset of $\mc U_a \times \PP^{r-1}_R$, given by several equations, among which are
		\begin{equation}\label{eq:blow-up}
			f_j^{(a)}x_k = f_k^{(a)}x_j \qquad \text{for } 1 \leq j,k \leq r.
		\end{equation}
		
		On the open set $\mc V_{ak} :=\mc V_a \cap \{x_k \neq 0\}$, the equation of the exceptional divisor (that is, the function representing the section $s$ on this open set) is then~$f_k^{(a)} = 0$. 		
		
		Using equations~(\ref{eq:blow-up}), for all $a\in A$, $k \leq r$ we have 
		\begin{align*}
			&&\ddiv[\nu]{f_j^{(a)}}{\mc V_{ak}} & = \ddiv[\nu]{f_k^{(a)}}{\mc V_{ak}} + \ddiv[\nu]{x_j}{\mc V_{ak}} & \\
			&&&\geq \ddiv[\nu]{f_k^{(a)}}{\mc V_{ak}} & \\
			\text{hence} && \inf_j \left( \ddiv[\nu]{f_j^{(a)}}{\mc V_{ak}}\right) & = \ddiv[\nu]{f_k^{(a)}}{\mc V_{ak}} = \ddiv[\nu]{s}{\mc V_{ak}} &\\
			\text{and thus} && \inf_j \left( \ddiv[\nu]{f_j^{(a)}}{\mc V_a}\right) & = \ddiv[\nu]{s}{\mc V_a}
		\end{align*}
		Applying proposition~\ref{prop:min_courants}, we obtain
		\begin{align*}
			(u_a)_{0*} \ddiv[\nu]{s}{\mc V_a} & = \inf_j \left( \ddiv[\nu]{f_j^{(a)}}{\mc U_a}\right) 
		\shortintertext{but we also have}
			(u_a)_{0*} \ddiv[\nu]{s}{\mc V_a} & = (u_a)_{0*}(\iota'_a)_0^* \ddiv[\nu]{s}{\t \XX} \\ 
			& = (\iota_a)_0^* u_{0*} \ddiv[\nu]{s}{\t \XX}  
		\end{align*}
		where we used \parencite[prop. 1.7]{fulton--IntersectionTheory1998} for the last equality, the map~$u_a$ being proper since it is a blow-up map.
		
		We finally obtain
		 \[(\iota_a)_0^* (\pi_* \divnu{s})_\XX = (\iota_a)_0^* u_{0*} \ddiv[\nu]{s}{\t \XX} = \inf_j \left( \ddiv[\nu]{f_j^{(a)}}{\mc U_a}\right).\] 
		This is precisely what we wanted to show, and this concludes the proof.
	\end{proof}

\section{Levine formula and other corollaries}
	
	Let us now focus on a particular case of theorem~\ref{thm:bost-g-s}, where the rank $r$ vector bundle $F$ is of the form $L \oplus \dots \oplus L$, with $L$ a line bundle on $X$. We also assume that the metric $\mscr F$ splits accordingly: there exists a metric $\mscr L$ of $L$ such that $\mscr F_\XX = \mscr L_\XX \oplus \dots \oplus \mscr L_\XX$ on compatible models.
	
	By applying theorem~\ref{thm:bost-g-s} in this case, we will recover an analog of the generalized Levine formula (theorem \ref{thm:levine_gh_complex}), which in turns implies both the usual Levine formula (when $L = \OO(1)$ on $X = \PP^n$) and Martinelli formula (when $L = \OO_X$).
	
	\medskip
	
	In this situation, the global section $\sigma$ can be written as $\sigma = (s_1, \dots, s_r)$, with the $s_i$ being global sections of $L$ and forming a regular sequence, and the closed subset $Z$ is equal to ${\{ s_1 = \dots = s_r = 0 \}}$.
	
	Let us still write $\pi \colon \t X \to X$ for the blow-up of $X$ along $Z$, $E$ the exceptional divisor and $f$ the canonical closed immersion $\t X \into \PP(F)$. In this situation, the projective space $\PP(F)$ can be described as follows:
	\begin{align*}
		\PP(F) & = \Proj \big(\Sym(\bigoplus_{k=0}^{r-1} L^\vee)\big) = \Proj \big( \Sym (\bigoplus_{k=0}^{r-1} \OO_X)\proddot L^\vee\big)  \\ 
		& \simeq \Proj \big(\Sym (\bigoplus_{k=0}^{r-1} \OO_X)\big) = X \times_R \PP^{r-1}_R
	\end{align*}
	where the isomorphism comes from \textcite[lemma~II.7.9] {hartshorne--AlgebraicGeometry1977}. Note that modulo this isomorphism, we have 
	\begin{equation}\label{eq:o_F(1)}
		\metr{\OO_F(1)} = p^*\metr{L^\vee} \otimes q^* \metr{\OO_{\PP_R^{r-1}}(1)}
	\end{equation}
	where $p$ and $q$ are the two projections from $X \times \PP^{r-1}_R$. The notations are summarized in this diagram:
	\[\begin{tikzcd}
		{\t X} & {X \times \PP^{r-1}_K} \\
		X & {\PP^{r-1}_K}
		\arrow["f", hook, from=1-1, to=1-2]
		\arrow["\pi"', from=1-1, to=2-1]
		\arrow["p", from=1-2, to=2-1]
		\arrow["q"', from=1-2, to=2-2]
	\end{tikzcd}\]

	As pointed out in the introduction, to state an analog of Levine formula we need to construct the differential forms on the blow-up and not on $X$ itself, as we lack a notion of differential forms on proper open subsets. 
	This is however of little consequence, because the form $\om_0$ appearing in the original statement satisfies 
	\[[\om_0] = \pi_* [f^* q^* c_1(\metr{\OO_{\PP_\CC^{r-1}}(1)})].\] 
	
	This allows us to state the following analog of Levine formula in our non-archimedean setup.
	
	\begin{thm}\label{thm:levine_gh}
		In the above situation, let $s$ be the canonical section of $\OO_{\t X}(E)$, and let $\alpha = \pi^* c_1(\metr{L})$ and $\beta = f^*q^* c_1(\metr{\OO_{\PP^{r-1}_K}(1)})$, they are elements of~$\cform{1}(\t X)$.
		
		Let $\Lambda$ be the current in $\mcurr{r-1}(X)$ defined by
		\[\Lambda = \pi_*\left(\divnu{s} \wedge \sum_{k=0}^{r-1} \alpha^{r-1-k}\wedge\beta^k\right).\]
		Then \[\ddc \Lambda + \delta_Z = [c_1(\metr{L})^r].\]
	\end{thm}

	\begin{proof}
		Using theorem~\ref{thm:bost-g-s}, it is enough to show the two following equalities:
		\begin{align*}
			c_r(\metr{F}) & = c_1(\metr{L})^r ~;\\
			f^* c_{r-1}(\metr{Q_F}) & = \sum_{k=0}^{r-1} \alpha^{r-1-k}\wedge\beta^k.
		\end{align*}
		
		First note that by lemma~\ref{lm:exact sq} and a straightforward induction we obtain $c_\sbullet(\metr{F}) = c_\sbullet(\metr{L})^r$, that is, for all $k \geq 0$, 
		\begin{equation}\label{eq:c_K(F)}
			c_k(\metr{F}) = \binom{r}{k} c_1(\metr{L}) ^k,
		\end{equation}
		which takes care of the first term to be computed.
		
		Then using equation~(\ref{eq:c_l(Q)}), we compute:
		\begin{align*}
			f^*c_{r-1}(\metr{Q_F}) 
			& = \sum_{j=0}^{r-1} f^*p^*c_{r-1-j}(\metr{F}) \wedge f^* c_1(\metr{\OO_F(1)})^j \\
			& = \sum_{j=0}^{r-1} \binom{r}{r-1-j} \alpha^{r-1-j} \wedge (\beta - \alpha)^j \qquad \qquad \  \text{ using (\ref{eq:c_K(F)}) and (\ref{eq:o_F(1)})} \\
			& = \sum_{j=0}^{r-1} \binom{r}{r-1-j} \alpha^{r-1-j} \wedge \sum_{k=0}^j \binom{j}{k} \beta^k \alpha^{j-k} (-1)^{j-k} \\
			& = \sum_{k=0}^{r-1} \alpha^{r-1-k}\beta^k \sum_{j=k}^{r-1} \binom{r}{r-1-j}\binom{j}{k}(-1)^{j-k}
		\end{align*}
		
		It remains to show that for all $0 \leq k \leq r-1$, the inner sum is equal to 1.
		Notice that for $x$ a formal variable,
		\[x^{r-1}\sum_{j=k}^{r-1} \binom{r}{r-1-j}\binom{j}{k}(-1)^{j-k}
 		 = \sum_{\substack{i+j = r-1 \\ j \geq k}} \binom{r}{i}x^i \binom{j}{k}x^j(-1)^{j-k}\]
		 is the $x^{r-1}$ term in the formal expression
		 \begin{align*}
		 	(1+x)^r\left(\sum_{j\geq k} \binom{j}{k}x^j(-1)^{j-k}\right)  
		 	& = (1+x)^r\frac{x^k}{k!}\sum_{j\geq k}j(j-1)\dots(j-k+1)(-x)^{j-k} \\ 
		 	& = (1+x)^r \frac{x^k}{k!} \frac{\dd^k}{\dd x^k} \sum_{j \geq 0} (-x)^j \\ 
		 	& = (1+x)^r \frac{x^k}{k!} \frac{k!}{(1+x)^{k+1}} \\ 
		 	& = x^k(1+x)^{r-1-k}.
		 \end{align*}
		 Therefore the wanted sum is the $x^{r-1}$ term of $x^k(1+x)^{r-1-k}$, that is, 1, as was to be shown. This concludes the proof.
	\end{proof}
	
	\begin{rk}
		In this situation, we have an easy description of the current $\divnu{\sigma}$: it is equal to $\inf \left( \divnu{s_0}, \dots, \divnu{s_{r-1}} \right)$, as can be seen by unfolding the respective definitions of $\divnu{\sigma}$ and $\divnu{s_i}$.
	\end{rk}
		
		We conclude by stating the two following corollaries.
	\begin{cor}\hfill
		\begin{itemize}
			\item If in the statement of theorem~\ref{thm:levine_gh} we take $X = \PP^n_K$, with homogeneous coordinates $(x_k)$, $L = \mc \OO_X(1)$ with standard metric and $s_k = x_k$, then we recover an analog of the usual Levine formula.
			
			\item If in the statement of theorem~\ref{thm:levine_gh} we take $L$ to be~$\mc \OO_X$ with trivial metric, then $\alpha = c_1(\metr{\mc \OO_X}) = 0$, and the statement becomes the following analog of Martinelli formula:
			\begin{align*}
				\ddc \pi_* \left( \divnu{s} \wedge \beta^{r-1}\right) + \delta_Z = 0.
			\end{align*}
		\end{itemize}
	\end{cor}
	\printbibliography
	\address
	\end{document}